\documentclass[11pt,centertags,oneside]{amsart}
\usepackage{amsmath,amstext,amsthm,amscd,typearea,hyperref}
\usepackage{amssymb}
\usepackage{a4wide}
\usepackage[mathscr]{eucal}
\usepackage{mathrsfs}
\usepackage{typearea}
\usepackage{charter}
\usepackage{pdfsync}
\usepackage{graphicx}

\usepackage{caption}

\usepackage{mathptmx}
\usepackage[per-mode=symbol]{siunitx}
\usepackage{bm}

\usepackage[a4paper,width=16.5cm,top=3cm,bottom=3cm]{geometry}

\numberwithin{equation}{section}

\allowdisplaybreaks
\emergencystretch=\maxdimen
\usepackage{multicol}
\usepackage{xcolor}

\newtheorem{theorem}{Theorem}[section]

\newtheorem{proposition}[theorem]{Proposition}
\newtheorem{corollary}[theorem]{Corollary}

\newcommand{\RR}{{\mathbb{R}}}

\newcommand{\SL}{\mathrm{SL}}

\newcommand{\SO}{{\mathrm{SO}}}

\newcommand{\Mat}{\mathrm{Mat}}

\newcommand{\ad}{{\mathrm{ad}}}
\newcommand{\Lie}[1]{\mathfrak{\lowercase{#1}}}

\newcommand{\tr}{\mathrm{tr}}

\providecommand{\RR}{\mathbb{R}}

\definecolor{han}{rgb}{1.0, 0, 0}

\newcommand{\rddots}{\reflectbox{$\ddots$}}

\newcommand{\rem}[1]{ }

\makeatletter
\newcommand\xleftrightarrow[2][]{%
  \ext@arrow 9999{\longleftrightarrowfill@}{#1}{#2}}
\newcommand\longleftrightarrowfill@{%
  \arrowfill@\leftarrow\relbar\rightarrow}
\makeatother

\title{Weight decomposition of $\Lie{sl}_d(\RR)$ with respect to the adjoint representation of $\Lie{so}(p,q)$}

\author{Jiyoung Han}
\address{School of Mathematics, Korea Institute for Advanced Study}
\email{jiyounghan@kias.re.kr, hanjiwind@gmail.com}

\begin{document}

\begin{abstract}
In this concise article, we compute the weight decomposition of $\Lie{SL}_d(\RR)$ with respect to the adjoint representation of $\Lie{so}(p,q)$, where $d=p+q$ and demonstrate in detail that $\Lie{SL}_d(\RR)$ comprises two irreducible $\Lie{so}(p,q)$-invariant subspaces. This can be employed to establish the well-known fact that the identity component of $\SO(p,q)$ is a maximal connected subgroup of $\SL_d(\RR)$.
\end{abstract}

\clearpage\maketitle
\thispagestyle{empty}


\medskip


\setcounter{tocdepth}{1}

\tableofcontents

\section{Introduction}\label{Sec: Introduction}

Let $\Lie{g}$ be a Lie algebra and consider a Lie algebra representation $(\Phi, V)$ of $\Lie{g}$. A subspace $W\subseteq V$ is called \emph{$\Lie{g}$-invariant} if $\Phi(X)W\subseteq W$ for all $X\in \Lie{g}$, and we say that a $\Lie{g}$-invariant subspace $W$ is \emph{$\Lie{g}$-irreducible} if there is no proper nontrivial subspace $W'\subset W$ which is $\Lie{g}$-invariant.
When $\Lie{g}$ is a semisimple real Lie algebra, it possesses a restricted root system, thus exhibiting the property that any representation $V$ for $\Lie{g}$ decomposes into $\Lie{g}$-irreducible sub-representations.

In this article, our focus lies on the adjoint representation for the semisimple Lie algebra $\Lie{so}(p,q)$ on $\Lie{sl}_d(\RR)$, where $d=p+q$.
The theorem below is likely familiar to experts; however, given the callenge in locating references, we will present a detailed proof.

\begin{theorem}\label{Main Theorem}
Let $d=p+q\ge 2$ with $p,q\ge 0$. The subalgebra $\Lie{sl}_d(\RR)$ consists of two $\Lie{so}(p,q)$-irreducible subspaces.
\end{theorem}

As a corollay, we can re-obtain the well-known result regarding the maximality of $\SO(p,q)$ on $\SL_d(\RR)$.

\begin{corollary}
The group $\SO(p,q)^\circ$ is maximal among connected subgroups of $\SL_d(\RR)$, where $d=p+q\ge2$
\end{corollary}

The paper is organized as follows. In Section 2, we revisit the restricted root system of $\Lie{so}(p,q)$. Based on this root system, we obtain all of weights and corresponding weight vectors of the compliment of $\Lie{so}(p,q)$ in $\Lie{sl}_d(\RR)$ in Section 3. For readers seeking a quick overview, we display root and weight vectors for a concrete example, $\Lie{so}(4,2)$, in Appendix. Finally, we present the proof of Theorem~\ref{Main Theorem} in Section 4.

\subsection*{Acknowledgements}
I would like to be thankful to Seonhee Lim and Keivan Mallahi-Karai for their encouragement throughout this project. The author acknowledges support from a KIAS Individual Grant MG088401 at Korea Institute for Advanced Study.

\section{Restricted Root Space Decomposition of $\Lie{so}(p,q)$}\label{Sec: Root Decomposition}

Let us briefly recall the restricted root system for $\Lie{so}(p,q)$, as introduced in \cite[Page 372-373]{Knapp}.
Denote elements of $\Lie{sl}_d(\RR)$ by two by two block matrix
\[
X=\left(\begin{array}{cc}
A & B\\
C & D \end{array}\right)
\]
with $A\in \Mat_p(\RR)$, $B\in \Mat_{p,q}(\RR)$, $C\in \Mat_{q,p}(\RR)$, and $D\in \Mat_q(\RR)$,  where $p, q\ge 0$ and $p+q=d$.
For the sake of convenience, we will take basis of $\Mat_d(\RR)$ as $A_{i,j}$, $B_{i,j}$, $C_{i,j}$ and $D_{i,j}$ as follows: if we let $\{E_{i,j}\}_{1\le i,j\le d}$ be a standard basis of $\Mat_d(\RR)$, 
\[\begin{array}{ll}
A_{i,j}=E_{p+1-i,p+1-j}, & B_{i,j}=E_{p+1-i,p+j}\;\text{and}\\
C_{i,j}=E_{p+i, p+1-j}, & D_{i,j}=E_{p+i, p+j}.
\end{array}
\]

The Lie algebra $\Lie{so}(p,q)$ consists of matrices $X$ in $\Lie{sl}_d(\RR)$ for which
$A$, $D$ are skew-symmetric and $B^t=C$, so that $\dim\Lie{so}(p,q)=\frac 1 2 (p(p-1)+q(q-1))+pq$.
There is a natural involution $\theta$ in $\Lie{so}(p,q)$ defined as the matrix transpose, leading the Cartan decomposition $\Lie{so}(p,q)=\Lie{k}\oplus \Lie{p}$, where $\Lie{k}$ denotes the set of $X\in \Lie{so}(p,q)$ for which $B$ (and consequently $D$) is a zero matrix, and $\Lie{p}$ represents the set of those $X$ whose $A$- and $D$-components are zero matrices.
Then one can take a maximal abelian subalgebra $\Lie{a}$ of $\Lie{p}$ whose elements are of the form
\begin{equation}\label{matrix F}
F=\left(\begin{array}{c|c}
 & \hspace{-0.08in}\begin{array}{ccc}
  & & \\
  & & a_{q}\\
  & \rddots & \\
  a_1 & & \end{array}\hspace{-0.08in}\\
  \hline
  \hspace{-0.08in}\begin{array}{cccc}
   & & & a_1 \\
   & & \rddots & \\
   & a_q& &
  \end{array}\hspace{-0.08in}
  & 
\end{array}\right)=\sum_{i=1}^q a_i (B_{i,i}+C_{i,i}).
\end{equation}
Here, withought loss of generality, we assume that $p\ge q$. Denote by $$\{f_i: i=1,\ldots, q\}$$ the dual basis of $\Lie{a}$ given by $f_i(B_{j,j}+C_{j,j})=\delta_{ij}$.
We say that a nonzero $\lambda\in \Lie{a}^*$ is a restricted root for $(\Lie{g}, \Lie{a})$ if there is a nonzero $X\in \Lie{g}$ such that for all $F\in \Lie{a}$,
\[
[F, X]=\lambda(F)X.
\]

\begin{proposition}\label{Thm: Root}
The restricted roots of $(\Lie{so}(p,q), \Lie{a})$ are $\pm f_i$ for $1\le i\le q$ (if $p\neq q$), $\pm f_i\pm f_j$, and $\pm f_i \mp f_j$  for $1\le i< j\le q$.
\begin{center}
\begin{tabular}{c|c|c|c}
\hline
 & \# of this type & $\dim$ of each root space & \\
 \hline
 $\pm f_i$ & 2q & p-q & 2q(p-q) \\[0.05in]
 $\pm f_i \pm f_j$ & q(q-1)& 1 & q(q-1) \\[0.05in]
 $\pm f_i \mp f_j$ & q(q-1) & 1 & q(q-1) \\[0.05in]
 0 & 1 & $\frac 1 2 (p-q)(p-q-1)+ q$ & $\frac 1 2 (p-q)(p-q-1)+ q$\\[0.05in]
 \hline 
 & & & $\frac 1 2 (p(p-1)+q(q-1))+pq=\dim\Lie{so}(p,q)$\\ \hline
\end{tabular}
\end{center}
(Here, notice that $0$ is not strictly a restricted root; however, we include it for convenience.)

For each restricted root (including $0$), a basis of the root space $\Lie{so}(p,q)_\lambda$ can be taken as follows. See  also Appendix for the matrix form in case of $\Lie{so}(4,2)$.
\begin{enumerate}
\item[i)] $\lambda=\pm f_i$ ($1\le i\le q$).

For $\lambda=f_i$, $\Lie{so}(p,q)_\lambda$ is generated by
\[
H(\lambda)_\ell=A_{q+\ell, i}-B_{q+\ell, i}-A_{i, q+\ell}-C_{i, q+\ell},\; 1\le \ell \le p-q,
\]
and for $\lambda=-f_i$,
\[
\Lie{so}(p,q)_\lambda=
\left\langle H(\lambda)_\ell=A_{q+\ell, i}+B_{q+\ell, i}-A_{i, q+\ell}+C_{i, q+\ell}: 1\le \ell \le p-q \right\rangle.
\]

\item[ii)] $\lambda=\pm f_i\pm f_j$ ($1\le i<j\le q$).

For $\lambda=f_i+f_j$,
\[
\Lie{so}(p,q)_\lambda=\RR.\left(-A_{i,j}+B_{i,j}+A_{j,i}-B_{j,i}-C_{i,j}+D_{i,j}+C_{j,i}-D_{j,i}\right),
\]
and for $\lambda=-f_i-f_j$,
\[
\Lie{so}(p,q)_\lambda=\RR.\left(-A_{i,j}-B_{i,j}+A_{j,i}+B_{j,i}+C_{i,j}+D_{i,j}-C_{j,i}-D_{j,i}\right).
\]

\item[iii)] $\lambda=\pm f_i \mp f_j$ ($1\le i<j\le q$).

For $\lambda=f_i-f_j$,
\[
\Lie{so}(p,q)_\lambda=\RR.\left(-A_{i,j}-B_{i,j}+A_{j,i}-B_{j,i}-C_{i,j}-D_{i,j}-C_{j,i}+D_{j,i}\right),
\]
and for $\lambda=-f_i+f_j$,
\[
\Lie{so}(p,q)_\lambda=\RR.\left(-A_{i,j}+B_{i,j}+A_{j,i}+B_{j,i}+C_{i,j}-D_{i,j}+C_{j,i}+D_{j,i}\right).
\]

\item[iv)] $\lambda=0$.

Following the notation of \cite{Knapp}, $\Lie{so}(p,q)_0=\Lie{m}\oplus \Lie{a}$, where $$\Lie{m}=\sum_{1\le i<j\le p-q} \RR.(A_{q+i,-q+j}-A_{q+j, -q+i})$$ is the centralizer of $\Lie{a}$ in $\Lie{k}$, which is isomorphic to $\Lie{so}(p-q)$ in our case.
\end{enumerate}
\end{proposition}

\section{Weight Decomposition of $\Lie{sl}_d(\RR)$}\label{Sec: Weight Decomposition}

Define
\begin{equation}\label{Lie algebra s}
\Lie{s}=\left\{X=\left(\begin{array}{cc}
A & B\\
C & D \end{array}\right):\begin{array}{c}
\tr(A)+\tr(B)=0,\\
A=A^t,\; D=D^t\\
B=-C^t.
\end{array}\right\}.
\end{equation}
It is not hard to check that $\Lie{sl}_d(\RR)=\Lie{so}(p,q)\oplus \Lie{s}$, and $\Lie{s}$ is invariant under the adjoint action of $\Lie{so}(p,q)$.

Following the approach in Section~\ref{Sec: Root Decomposition}, we will present the weight decomposition of $\Lie{s}$ for the adjoint representation of $\Lie{so}(p,q)$ below, providing a comprehensive overview of the weight decomposition of $\Lie{sl}_d(\RR)$.

\begin{proposition}\label{Thm: Weight}
The weights of $\Lie{s}$ for the adjoint representation of $\Lie{so}(p,q)$ are $\pm f_i$ (if $p\neq q$) and $\pm 2f_i$ for $1\le i\le q$, $\pm f_i\pm f_j$ and $\pm f_i \mp f_j$  for $1\le i< j\le q$ and $0$.
\begin{center}
\begin{tabular}{c|c|c|c}
\hline
 & \# of this type & $\dim$ of each root space & \\
 \hline
 $\pm f_i$ & 2q & p-q & 2q(p-q) \\[0.05in]
 $\pm 2f_i$ & 2q & 1 & 2q \\[0.05in]
 $\pm f_i \pm f_j$ & q(q-1)& 1 & q(q-1) \\[0.05in]
 $\pm f_i \mp f_j$ & q(q-1) & 1 & q(q-1) \\[0.05in]
 0 & 1 & $\frac 1 2 (p-q)(p-q+1)+ q-1$ & $\frac 1 2 (p-q)(p-q+1)+ q-1$\\[0.05in]
 \hline 
 & & & $\frac 1 2 (p(p+1)+q(q+1))+pq-1=\dim\Lie{s}$\\ \hline
\end{tabular}
\end{center}

For each weight $\omega$, a basis of the weight space $\Lie{s}_\omega$ can be taken as follows (see Appendix for the matrix form in case of $\Lie{so}(4,2)$).
\begin{enumerate}
\item[i)] $\omega=\pm f_i$ ($1\le i\le q$).

For $\omega=f_i$, $\Lie{s}_\omega$ is generated by
\[
S(\omega)_\ell=A_{q+\ell, i}-B_{q+\ell, i}+A_{i, q+\ell}+C_{i, q+\ell},\; 1\le \ell \le p-q,
\]
and for $\omega=-f_i$,
\[
\Lie{s}_\omega=
\left\langle S(\omega)_\ell=A_{q+\ell, i}+B_{q+\ell, i}+A_{i, q+\ell}-C_{i, q+\ell}: 1\le \ell \le p-q \right\rangle.
\]

\item[ii)] $\omega=\pm 2f_i$ ($1\le i \le q$).

For $\omega=2f_i$,
\[
\Lie{s}_\omega=
\RR.\left(A_{i,i}-B_{i,i}+C_{i,i}-D_{i,i}\right),
\] 
and for $\omega=-2f_i$,
\[
\Lie{s}_\omega=
\RR.\left(A_{i,i}+B_{i,i}-C_{i,i}-D_{i,i}\right).
\]

\item[iii)] $\omega=\pm f_i\pm f_j$ ($1\le i<j\le q$).

For $\omega=f_i+f_j$,
\[
\Lie{s}_\omega=\RR.\left(-A_{i,j}+B_{i,j}-A_{j,i}+B_{j,i}-C_{i,j}+D_{i,j}-C_{j,i}+D_{j,i}\right),
\]
and for $\lambda=-f_i-f_j$,
\[
\Lie{s}_\omega=\RR.\left(A_{i,j}+B_{i,j}+A_{j,i}+B_{j,i}-C_{i,j}-D_{i,j}-C_{j,i}-D_{j,i}\right),
\]

\item[iv)] $\lambda=\pm f_i \mp f_j$ ($1\le i<j\le q$).

For $\lambda=f_i-f_j$,
\[
\Lie{s}_\omega=\RR.\left(A_{i,j}+B_{i,j}+A_{j,i}-B_{j,i}+C_{i,j}+D_{i,j}-C_{j,i}+D_{j,i}\right).
\]

and for $\lambda=-f_i+f_j$,
\[
\Lie{s}_\omega=\RR.\left(A_{i,j}-B_{i,j}+A_{j,i}+B_{j,i}-C_{i,j}+D_{i,j}+C_{j,i}+D_{j,i}\right).
\]

\item[v)] $\omega=0$.

The weight space $\Lie{s}_0$ is composed of the first $(p-q)\times (p-q)$ symmetric matrices and diagonal matrices. One can take a basis of $\Lie{s}_0$ as follows.
\[\begin{gathered}
A_{q+i, q+j}+A_{q+j, q+i}\; (1\le i<j\le p-q),\\
2A_{q+i,q+i}+(A_{1,1}+D_{1,1})\; (1\le i\le p-q),\;\text{and}\\
(A_{i,i}+D_{i,i})-(A_{i+1, i+1}+D_{i+1,i+1})\; (1\le i\le q-1).
\end{gathered}\]

\end{enumerate}
\end{proposition}
\begin{proof}
i) For $F\in \Lie{a}$ as in \eqref{matrix F}, by putting $\ell'=\ell+q$,
\[\begin{split}
&\ad(F)(S(f_i)_\ell)=[F, S(f_i)_\ell]\\
&=\left(\sum_{k=1}^q a_k(B_{k,k}+C_{k,k})\right)\left(A_{q+\ell, i}-B_{q+\ell, i}+A_{i, q+\ell}+C_{i, q+\ell}\right)\\
&\hspace{0.8in}-\left(A_{q+\ell, i}-B_{q+\ell, i}+A_{i, q+\ell}+C_{i, q+\ell}\right)\left(\sum_{k=1}^q a_k(B_{k,k}+C_{k,k})\right)\\
&=\sum_{k=1}^q a_k\left[(E_{p+k, p+1-k}+E_{p+1-k,p+k})(E_{p+1-\ell',p+1-i}-E_{p+1-\ell',p+i}+E_{p+1-i,p+1-\ell'}+E_{p+i,p+1-\ell'})\right.\\
&\hspace{0.65in}\left.-(E_{p+1-\ell',p+1-i}-E_{p+1-\ell',p+i}+E_{p+1-i,p+1-\ell'}+E_{p+i,p+1-\ell'})(E_{p+k, p+1-k}+E_{p+1-k,p+k})\right]\\
&=a_i(E_{p+1-\ell',p+1-i}-E_{p+1-\ell',p+i}+E_{p+1-i,p+1-\ell'}+E_{p+i,p+1-\ell'})\\
&=f_i(F)\left(A_{q+\ell, i}-B_{q+\ell, i}+A_{i, q+\ell}+C_{i, q+\ell}\right)
=f_i(F)S(f_i)_\ell.
\end{split}\]

The remaining cases can be verified in a similar manner. Given that the total number of linearly independent weight vectors is equal to  $\dim\Lie{s}$,  we can thus derive the weight decomposition of $\Lie{s}$.
\end{proof}

\section{Proof of Main Theorem}\label{Sec: Proof}
Theorem~\ref{Main Theorem} follows directly from the proposition below.

\begin{proposition}
The subspace $\Lie{s}$ defined in \eqref{Lie algebra s} is $\Lie{so}(p,q)$-irreducible.
\end{proposition}
\begin{proof}
It is sufficient to show that any weight vectors of $\Lie{s}$ provided in Proposition~\ref{Thm: Weight} is mapped to another weight vector by the adjoint action of $\Lie{so}(p,q)$. Let $[H(\lambda)]$ and $[S(\omega)]$ are the class of root vectors for $\lambda$ and weight vectors for $\omega$, respectively, if the root or weight space is one-dimensional, and $H(\lambda)_\ell$ and $S(\omega)_\ell$ are as in Proposition~\ref{Thm: Root} and Proposition~\ref{Thm: Weight}, respectively. 
The following, for instance, demonstrates the transitivity of the set of all nonzero weights.
\begin{itemize}
\item $[S(2f_1)] \xleftrightarrow[H(-f_1\mp f_j)]{H(f_1\mp f_j)} [S(f_1\pm f_j)] \xleftrightarrow[H(-f_1\mp f_j)]{H(f_1\pm f_j)} [S(\pm 2f_j)]$;
\item $[S(\mp f_i \pm f_j)]\xleftrightarrow[H(\pm f_j+f_1)]{H(\mp f_i-f_1)}[S(f_1\pm f_j)]\xleftrightarrow[H(\pm f_i-f_1)]{H(\mp f_i+f_1)} [S(\pm f_i\pm f_j)]\xleftrightarrow [H(-f_1\mp f_i)]{H(f_1\pm f_i)}[S(-f_1\pm f_j)]$;
\item $[S(\pm 2f_i)]\xleftrightarrow[H(\mp f_i)_\ell]{H(\pm f_i)_\ell}[S(\pm f_i)_\ell]$
\end{itemize}
It can be verified that each mapping sends a nonzero weight vectors to another nonzero weight vector. 
As for the connection to weight vectors of zero weight, we have the following observaton:
\begin{itemize}
\item $[S(-f_1)_j]\xleftrightarrow[H(f_1)_i]{H(-f_1)_i} [A_{q+i, q+j}+A_{q+j, q+i}]$
\item $[S(-f_1)_i]\xleftrightarrow[H(f_1)_i]{H(-f_1)_i} [2A_{q+i,q+i}+(A_{1,1}+D_{1,1})]$
\item $[S(-f_i-f_{i+1})]\xleftrightarrow[H(f_i+f_{i+1})]{H(-f_i-f_{i+1})} [A_{i,i}+D_{i,i}-(A_{i+1,i+1}+D_{i+1, i+1})]$.
\end{itemize}
\end{proof}

\newpage
\section*{Appendix}\label{Sec: Appendix}
We now illustrate the matrix form of root and weight vectors discussed in Theorem~\ref{Thm: Root} and Theorem~\ref{Thm: Weight}, respectively, in the case of $\Lie{so}(4,2)$.

We have that $\Lie{so}(4,2)_0=\Lie{a}\oplus \Lie{m}$, where
\[
\Lie{a}=\left\{{\small \left(\hspace{-0.1in}\begin{array}{cccc|cc}
&&&&&\\
&&&&&\\
&&&&&a_2\\
&&&&a_1&\\
\hline
&&&a_1&&\\
&&a_2&&&
\end{array}\hspace{-0.1in}\right)}: a_1, \; a_2\in \RR \right\}
\quad\text{and}\quad
\Lie{m}=\left\{{\small \left(\hspace{-0.1in}\begin{array}{cccc|cc}
&x&&&&\\
-x&&&&&\\
&&&&&\\
&&&&&\\
\hline
&&&&&\\
&&&&&
\end{array}\hspace{-0.1in}\right)}: x\in \RR\right\}.
\]

The restricted roots for $(\Lie{so}(p,q), \Lie{a})$ and corresponding root spaces are as follows.

\begin{center}
\begin{table}[h]
\begin{tabular}{|c|c|c|}
\hline
root &$f_1$ & $f_2$ \\
\hline
$\begin{array}{c}
\text{root}\\
\text{vectors}
\end{array}$ & $\left(\begin{array}{cccc|cc}
&&&x_2&-x_2&\\
&&&x_1&-x_1&\\
&&&&&\\
-x_2&-x_1&&&&\\
\hline
-x_2&-x_1&&&&\\
&&&&&
\end{array}\right)$ & $\left(\begin{array}{cccc|cc}
&&x_2&&&-x_2\\
&&x_1&&&-x_1\\
-x_2&-x_1&&&&\\
&&&&&\\
\hline
&&&&&\\
-x_2&-x_1&&&&
\end{array}\right)$\\
\hline
& $-f_1$ & $-f_2$ \\
\hline
& $\left(\begin{array}{cccc|cc}
&&&x_2&x_2&\\
&&&x_1&x_1&\\
&&&&&\\
-x_2&-x_1&&&&\\
\hline
x_2&x_1&&&&\\
&&&&&
\end{array}\right)$ & $\left(\begin{array}{cccc|cc}
&&x_2&&&x_2\\
&&x_1&&&x_1\\
-x_2&-x_1&&&&\\
&&&&&\\
\hline
&&&&&\\
x_2&x_1&&&&
\end{array}\right)$\\
\hline
& $f_1+f_2$ & $-f_1-f_2$ \\
\hline
 & $\left(\begin{array}{cccc|cc}
&&&&&\\
&&&&&\\
&&&x&-x&\\
&&-x&&&x\\
\hline
&&-x&&&x\\
&&&x&-x&
\end{array}\right)$ & $\left(\begin{array}{cccc|cc}
&&&&&\\
&&&&&\\
&&&x&x&\\
&&-x&&&-x\\
\hline
&&x&&&x\\
&&&-x&-x&
\end{array}\right)$ \\
\hline
& $f_1-f_2$ & $-f_1+f_2$\\
\hline
 & $\left(\begin{array}{cccc|cc}
&&&&&\\
&&&&&\\
&&&x&-x&\\
&&-x&&&-x\\
\hline
&&-x&&&-x\\
&&&-x&x&
\end{array}\right)$ & $\left(\begin{array}{cccc|cc}
&&&&&\\
&&&&&\\
&&&x&x&\\
&&-x&&&x\\
\hline
&&x&&&-x\\
&&&x&x&
\end{array}\right)$\\
\hline
\end{tabular}
\caption{The restricted root system of $\Lie{so}(4,2)$}
\end{table}
\end{center}

\newpage
And the following table exhibits weight spaces of nonzero weights in $\Lie{s}$ for $(\Lie{so}(4,2), \Lie{a})$.
\begin{table}[h]
\begin{center}
\begin{tabular}{|c|c|c|c|}
\hline
weight &$f_1$ & $f_2$ & $-f_1$\\
\hline
$\hspace{-0.1in}\begin{array}{c}
\text{weight}\\
\text{vectors}
\end{array}\hspace{-0.1in}$ & {\small $\left(\hspace{-0.1in}\begin{array}{cccc|cc}
&&&x_2&-x_2&\\
&&&x_1&-x_1&\\
&&&&&\\
x_2&x_1&&&&\\
\hline
x_2&x_1&&&&\\
&&&&&
\end{array}\hspace{-0.1in}\right)$} & {\small $\left(\hspace{-0.1in}\begin{array}{cccc|cc}
&&x_2&&&-x_2\\
&&x_1&&&-x_1\\
x_2&x_1&&&&\\
&&&&&\\
\hline
&&&&&\\
x_2&x_1&&&&
\end{array}\hspace{-0.1in}\right)$} &  {\small $\left(\hspace{-0.1in}\begin{array}{cccc|cc}
&&&x_2&x_2&\\
&&&x_1&x_1&\\
&&&&&\\
x_2&x_1&&&&\\
\hline
-x_2&-x_1&&&&\\
&&&&&
\end{array}\hspace{-0.1in}\right)$} \\
\hline
 &$-f_2$ & $2f_1$ & $2f_2$ \\
\hline
 &{\small $\left(\hspace{-0.1in}\begin{array}{cccc|cc}
&&x_2&&&x_2\\
&&x_1&&&x_1\\
x_2&x_1&&&&\\
&&&&&\\
\hline
&&&&&\\
-x_2&-x_1&&&&
\end{array}\hspace{-0.1in}\right)$} &  {\small $\left(\hspace{-0.1in}\begin{array}{cccc|cc}
&&&&&\\
&&&&&\\
&&&&&\\
&&&x&-x&\\
\hline
&&&x&-x&\\
&&&&&
\end{array}\hspace{-0.1in}\right)$} & {\small $\left(\hspace{-0.1in}\begin{array}{cccc|cc}
&&&&&\\
&&&&&\\
&&x&&&-x\\
&&&&&\\
\hline
&&&&&\\
&&x&&&-x
\end{array}\hspace{-0.1in}\right)$} \\
\hline
& $-2f_1$ & $-2f_2$ & $f_1+f_2$\\
\hline
& {\small $\left(\hspace{-0.1in}\begin{array}{cccc|cc}
&&&&&\\
&&&&&\\
&&&&&\\
&&&x&x&\\
\hline
&&&-x&-x&\\
&&&&&
\end{array}\hspace{-0.1in}\right)$} & {\small $\left(\hspace{-0.1in}\begin{array}{cccc|cc}
&&&&&\\
&&&&&\\
&&x&&&x\\
&&&&&\\
\hline
&&&&&\\
&&-x&&&-x
\end{array}\hspace{-0.1in}\right)$} &  {\small $\left(\hspace{-0.1in}\begin{array}{cccc|cc}
&&&&&\\
&&&&&\\
&&&-x&x&\\
&&-x&&&x\\
\hline
&&-x&&&x\\
&&&-x&x&
\end{array}\hspace{-0.1in}\right)$}\\
\hline
& $-f_1-f_2$ & $f_1-f_2$ & $-f_1+f_2$\\
\hline
 & {\small $\left(\hspace{-0.1in}\begin{array}{cccc|cc}
&&&&&\\
&&&&&\\
&&&x&x&\\
&&x&&&x\\
\hline
&&-x&&&-x\\
&&&-x&-x&
\end{array}\hspace{-0.1in}\right)$} & {\small $\left(\hspace{-0.1in}\begin{array}{cccc|cc}
&&&&&\\
&&&&&\\
&&&x&-x&\\
&&x&&&x\\
\hline
&&x&&&x\\
&&&-x&x&
\end{array}\hspace{-0.1in}\right)$} & {\small $\left(\hspace{-0.1in}\begin{array}{cccc|cc}
&&&&&\\
&&&&&\\
&&&x&x&\\
&&x&&&x\\
\hline
&&-x&&&x\\
&&&x&x&
\end{array}\hspace{-0.1in}\right)$}\\
\hline
\end{tabular}
\caption{The weight decomposition of $\Lie{s}$ for the adjoint action of $\Lie{so}(4,2)$}
\end{center}
\end{table}

The weight space $\Lie{s}_0$ of zero weight is generated by
\[
{\small \left(\hspace{-0.1in}\begin{array}{cccc|cc}
&x_1&&&&\\
x_1&&&&&\\
&&&&&\\
&&&&&\\
\hline
&&&&&\\
&&&&&
\end{array}\hspace{-0.1in}\right),
\left(\hspace{-0.1in}\begin{array}{cccc|cc}
2x_2&&&&&\\
&2x_3&&&&\\
&&&&&\\
&&&-x_2-x_3&&\\
\hline
&&&&-x_2-x_3&\\
&&&&&
\end{array}\hspace{-0.1in}\right)},\;\text{and}{\small \left(\hspace{-0.1in}\begin{array}{cccc|cc}
&&&&&\\
&&&&&\\
&&-x_4&&&\\
&&&x_4&&\\
\hline
&&&&x_4&\\
&&&&&-x_4
\end{array}\hspace{-0.1in}\right).}
\]


\end{document}